\documentclass[11pt]{article}
\usepackage[top=2.54cm, bottom=2.54cm, left=2.70cm, right=2.70cm]{geometry}
\usepackage[tbtags]{amsmath}
\usepackage{amssymb}
\usepackage{amsthm}
\usepackage{mathrsfs}
\usepackage{fancyhdr}
\usepackage{float}
\usepackage{graphicx}
\usepackage{tikz}
\usetikzlibrary{arrows.meta,calc,fit,backgrounds}
\usepackage{tkz-graph}
\usepackage{tkz-berge}
\usetikzlibrary{decorations.pathreplacing,decorations.markings}
\usepackage[numbers,sort&compress]{natbib}
\usepackage{setspace}
\usepackage{caption}
\usepackage{url}

\allowdisplaybreaks[4]
\newtheorem{theorem}{\bf Theorem}[section]

\newtheorem{proposition}[theorem]{Proposition}

\newtheorem{problem}[theorem]{Problem}
\newtheorem{corollary}[theorem]{Corollary}

\newtheorem{claim}{\indent Claim}

\DeclareMathOperator{\init}{init}
\DeclareMathOperator{\term}{term}
\DeclareMathOperator{\dist}{dist}
\DeclareMathOperator{\mader}{mader}
\DeclareMathOperator{\dom}{Dom}

\begin{document}
	
	\begin{spacing}{1.1}
		
\title{The exact digirth and out-degree bounds for $C(k,k)$-subdivision and out-forest in digraphs
}
		
		\author{Bin Chen$^a$, Jianfeng Hou$^b$, Xinmin Hou$^{c,d}$\thanks {Email address xmhou@ustc.edu.cn(X. Hou)},  Siyue Liu$^b$,  Xinjie Wang$^c$,  Xinyu Zhou$^e$\\
			\small$^a$ School of Mathematics and Statistics, Fuzhou University, Fujian, China\\
            \small$^b$ Center for Discrete Mathematics, Fuzhou University, Fujian, China\\
			\small$^c$ School of Mathematical Sciences,\\
			\small University of Science and Technology of China, Hefei 230026, Anhui, China\\
			\small$^d$ Hefei National Laboratory,\\
			\small University of Science and Technology of China, Hefei 230088, Anhui, China\\
			\small$^e$ College of Science,\\
			\small Southwest Petroleum University, Chengdu 610500, Sichuan, China}
		\date{}
		
		\maketitle
		
		\noindent{\bf Abstract}. For each integer \(k\ge2\), let \(C(k,k)\) denote the oriented cycle formed by two internally vertex-disjoint directed paths of length \(k\) with the same initial and terminal vertices. Picasarri-Arrieta and Rambaud ({\small European J. Combin., 2024}) asked for the smallest integer \(g_k\) such that every digraph with minimum out-degree at least two and digirth at least \(g_k\) contains a subdivision of \(C(k,k)\). We completely resolve this problem by proving that \(g_k=k+1\). Our proof relies on a refinement of Thomassen’s theorem ({\small Discrete Math., 1985}) on two directed cycles intersecting in exactly one vertex: we prove that every digraph with minimum out-degree at least two contains two directed cycles intersecting in exactly one arc. 
        We also establish an embedding theorem for out-forests with prescribed images of their roots, under minimum out-degree and digirth conditions determined by the number of leaves and the heights of vertices in antichains. As an application, for integers $k,\ell\ge2$, we show that every digraph with minimum out-degree at least \(k\) and digirth at least \(k(\ell-1)+1\) contains \(k\) directed paths of length \(\ell\) sharing only their initial vertex, which can be prescribed arbitrarily. This improves the previous digirth bound \((k^\ell-1)/(k-1)+1\) of Picasarri-Arrieta and Rambaud to a sharp bound.
		
		\noindent{\bf AMS}. 05C05; 05C20; 05C38.
		
		
		\section{Introduction}
		
		\noindent
		
		A central problem in structural digraph theory is to determine degree conditions that guarantee the existence of a subdivision of a given digraph. Mader~\cite{mader1985} conjectured that for every acyclic digraph $H$, there exists an integer $d(H)$ such that every digraph of minimum out-degree at least $d(H)$ contains a subdivision of $H$. Here an acyclic digraph $H$ is a digraph without directed cycles, and a subdivision of $H$ is a digraph obtained from $H$ by replacing its arcs by internally vertex disjoint directed paths.
		
		 Following the notation of Aboulker, Cohen, Havet, Lochet, Moura and  Thomass\'e~\cite{aboulker2019}, a digraph $H$ is called $\delta^+$-maderian if there exists an integer $d$ such that every digraph with minimum out-degree at least $d$ contains a subdivision of $H$. The least such integer is denoted by $\mader_{\delta^+}(H)$.
		
		An \emph{in-arborescence} is a rooted oriented tree whose arcs are all directed towards its root. Using the terminology above, Aboulker et al.~\cite{aboulker2019} proved that every nontrivial oriented path $P$ is $\delta^+$-maderian with $\mader_{\delta^+}(P)=|V(P)|-1$, and that every in-arborescence is $\delta^+$-maderian.
		
		A \emph{block} of an oriented cycle is a maximal directed subpath. For positive integers $k_1,k_2$ with $k_1+k_2\ge3$, let $C(k_1,k_2)$ denote the oriented cycle with two blocks of lengths $k_1$ and $k_2$. Aboulker et al.~\cite{aboulker2019} proved that $\mader_{\delta^+}(C(k_1,k_2))\le 2(k_1+k_2)-1$. For $k_1\ge k_2\ge2$, Gishboliner, Steiner, and Szab\'o~\cite{gishboliner2022} improved this bound to $k_1+3k_2-5$. Liu and Yoo~\cite{liu2026} subsequently determined the exact value $\mader_{\delta^+}(C(k_1,k_2))=k_1+k_2-1$.
		
		Picasarri-Arrieta and Rambaud~\cite{picasarri2024} developed a different perspective by imposing an additional digirth condition on the host digraph. 
        They proved the following theorem.
		
		\begin{theorem}[Picasarri-Arrieta and Rambaud~\cite{picasarri2024}]
			Let $k\ge2$, and let $D$ be a digraph satisfying
			$$
			\delta^+(D)\ge2
			\qquad\text{and}\qquad
			g(D)\ge8k-6.
			$$
			Then $D$ contains a subdivision of $C(k,k)$.
		\end{theorem}
		
		They further asked the following question.
        \begin{problem}[Problem 18 in \cite{picasarri2024}]\label{P1}
        Determine the smallest integer $g$ such that every digraph $D$ with $\delta^+(D)\ge 2$ and $g(D)\ge g$  contains a subdivision of $C(k,k)$.
		\end{problem}
	
        Chen, Hou, Ma, Yin, and Zhou~\cite{chen2026} subsequently improved the digirth bound from $8k-6$ to $4k+2$.
		
		\begin{theorem}[Chen et al.~\cite{chen2026}]
			Let $k\ge2$, and let $D$ be a digraph satisfying
			$$
			\delta^+(D)\ge2
			\qquad\text{and}\qquad
			g(D)\ge4k+2.
			$$
			Then $D$ contains a subdivision of $C(k,k)$.
		\end{theorem}
		
	In fact, the authors in \cite{picasarri2024} constructed  a digraph with minimum out-degree two and digirth $k$ containing no subdivision of $C(k,k)$. 
    Let $g_k$ be the smallest integer such that every digraph $D$ with $\delta^+(D)\ge2$ and $g(D)\ge g_k$ contains a subdivision of $C(k,k)$. Then $k+1\le g_k\le4k+2$.
In this paper, we prove that $g_k=k+1$, completely solving Problem~\ref{P1}.
    \begin{theorem}\label{ckk}
			Let $k\ge2$, and let $D$ and $v_0\in V(D)$ satisfy the degree conditions of Theorem~\ref{2-block}. If $g(D)\ge k+1$, then $D$ contains a subdivision of $C(k,k)$. Consequently, $g_k=k+1$.
	\end{theorem}

To prove Theorem~\ref{ckk}, we first establish the following structural theorem, which strength a result of Thomassen~\cite{thomassen1985} stating that every a digraph $D$ with $\delta^+(D)\ge2$ contains two directed cycles intersecting exactly in one vertex.
		 
		\begin{theorem}[Structural theorem]\label{2-block}
			Let $v_0$ be a vertex of a digraph $D$ such that
			$$
			d_D^+(v_0)\ge1,
			\qquad
			d_D^+(v)\ge2
			\quad\text{for every }v\in V(D)\setminus\{v_0\}.
			$$
			Then $D$ contains two directed cycles intersecting exactly in one arc, and both endpoints of the common arc are different from $v_0$.
		\end{theorem}

		Theorem~\ref{ckk} follows from Theorem~\ref{2-block} and the girth condition.

		\begin{proof}[Proof of Theorem~\ref{ckk}:]
			By Theorem~\ref{2-block}, $D$ contains two directed cycles whose intersection is exactly an arc $xy$. Deleting $xy$ from both cycles yields two internally vertex-disjoint $(y,x)$-paths, each of length at least $g(D)-1\ge k$. Their union is a subdivision of $C(k,k)$.
			
			Thus $g_k\le k+1$, which, together with the lower bound established above, gives $g_k=k+1$.
		\end{proof}
		
		Let $S_k^{(\ell)}$ denote the digraph consisting of $k$ directed paths of length $\ell$ sharing exactly their initial vertex. We call this common vertex the \emph{center} of $S_k^{(\ell)}$. Picasarri-Arrieta and Rambaud proved the following result~\cite{picasarri2024}.
		
		\begin{theorem}[Picasarri-Arrieta and Rambaud~\cite{picasarri2024}]\label{picasarri-star}
			Let $k\ge2$ and $\ell\ge1$, and let $D$ be a digraph satisfying
			$$
			\delta^+(D)\ge k
			\qquad\text{and}\qquad
			g(D)\ge \frac{k^\ell-1}{k-1}+1.
			$$
			Then, for every vertex $v\in V(D)$, the digraph $D$ contains a copy of $S_k^{(\ell)}$ with center $v$.
		\end{theorem}
		
		
		
		
		An \emph{out-arborescence} is a rooted oriented tree whose arcs are all directed away from its root, and an \emph{out-forest} is a disjoint union of out-arborescences. A vertex of an out-forest is called a \emph{leaf} if it has out-degree zero, and an \emph{internal vertex} otherwise. In particular, an isolated root is a leaf. For an out-forest $F$, let $L(F)$ denote its number of leaves. If $uv\in A(F)$, then $u$ is the \emph{parent} of $v$, and $v$ is a \emph{child} of $u$. We write $C_F(u)=N_F^+(u)$ for the set of children of $u$.
		
		Let $T$ be a component of $F$ with root $r$. If $w$ lies on the unique $(r,v)$-path in $T$, then $w$ is called an \emph{ancestor} of $v$, and $v$ is called a \emph{descendant} of $w$. In particular, every vertex is an ancestor and a descendant of itself. Two vertices are \emph{comparable} if one is an ancestor of the other, and otherwise they are \emph{incomparable}. In particular, vertices belonging to distinct components of $F$ are incomparable. An \emph{antichain} of $F$ is a set of pairwise incomparable vertices.
		
		For $v\in V(F)$, define the \emph{height} of $v$ by
		$$
		h_F(v)=\max\{\|P\|: P \text{ is a directed path from } v
		\text{ to a descendant leaf of } v\},
		$$
		where $\|P\|$ denotes the number of arcs of $P$. Thus, $h_F(v)$ is the maximum distance from $v$ to a leaf below it. In particular, $h_F(v)=0$ if $v$ is a leaf.
		
		We then define
		$$
		B(F)=
		\max_{A\text{ an antichain of }F}
		\sum_{v\in A}h_F(v).
		$$
		Thus $B(F)$ is the maximum total height over all antichains of $F$.
		
		We further establish the following result.
		
		\begin{theorem}\label{out-forest}
			Let $F$ be an out-forest with $L$ leaves and component roots $r_1,\ldots,r_s$, and let $D$ be a digraph. If
			\[
			\delta^+(D)\ge L,\qquad g(D)>B(F),
			\]
			then for any distinct vertices $\rho_1,\ldots,\rho_s\in V(D)$,
			there exists an embedding $\varphi:F\hookrightarrow D$ such that
			$\varphi(r_i)=\rho_i$ for every $i\in[s]$.
		\end{theorem}
		
        As an application of Theorem~\ref{out-forest}, we obtain the following improvement of Theorem~\ref{picasarri-star}. Further applications of Theorem~\ref{out-forest} are presented in Section~\ref{applications}.
      \begin{corollary}\label{Cor:improvement}
     Let integers $k,\ell\ge2$ and let $D$ be a digraph with $\delta^+(D)\ge k$ and $g(D)\ge k(\ell-1)+1$. Then, for every $v\in V(D)$, $D$ contains a copy of $S_k^{(\ell)}$ with center $v$.    
      \end{corollary}  

		\section{Preliminaries}
		
		\noindent
		
		\subsection{Notation}

		\noindent

		All digraphs considered in this paper are finite and have no loops or parallel arcs; pairs of opposite arcs are allowed. For terminology not defined here, we refer to~\cite{bangjensen2009}. For a nonnegative integer $k$, write $[k]=\{1,\ldots,k\}$, where $[0]=\varnothing$, and use $\sqcup$ to denote a disjoint union.

		Let $D=(V,A)$ be a digraph. Its order is $|D|=|V(D)|$, and its size is $\|D\|=|A(D)|$. We write $v\in D$ for $v\in V(D)$ and $uv\in D$ for $(u,v)\in A(D)$. For $S\subseteq V(D)$, let $D[S]$ denote the subdigraph induced by $S$, and put $D-S=D[V(D)\setminus S]$. When $S=\{v\}$, we write $D-v$ instead of $D-\{v\}$. A set of vertices is \emph{independent} if no arc has both endpoints in the set.

		For $v\in V(D)$, its \emph{out-neighborhood} and \emph{in-neighborhood} are
		$$
		N_D^+(v)=\{u\in V(D):vu\in A(D)\},
		\qquad
		N_D^-(v)=\{u\in V(D):uv\in A(D)\}.
		$$
		The vertices in these sets are called the \emph{out-neighbors} and \emph{in-neighbors} of $v$, respectively. The \emph{out-degree} and \emph{in-degree} of $v$ are $d_D^+(v)=|N_D^+(v)|$ and $d_D^-(v)=|N_D^-(v)|$. Their minimum and maximum values are denoted by $\delta^+(D),\delta^-(D),\Delta^+(D)$, and $\Delta^-(D)$, respectively. A vertex of in-degree zero is a \emph{source}, and a vertex of out-degree zero is a \emph{sink}. We omit the subscript $D$ when the digraph is clear from the context.

		For a set $S\subseteq V(D)$, define its \emph{external out-neighborhood} by
		$$
		N_D^+(S)=\left(\bigcup_{v\in S}N_D^+(v)\right)\setminus S.
		$$
		Thus, $N_D^+(S)$ consists only of vertices outside $S$. An arc \emph{leaves} $S$ if its initial vertex lies in $S$ and its terminal vertex lies outside $S$.

		An \emph{orientation} of a graph is obtained by assigning one direction to each edge. An \emph{oriented graph} is a digraph with at most one arc between each pair of distinct vertices. A \emph{directed walk} is a sequence $v_0,v_1,\ldots,v_m$ such that $v_{i-1}v_i\in A(D)$ for every $i\in[m]$; its length is $m$. It is \emph{closed} if $v_0=v_m$. A \emph{directed path}, or simply a \emph{path}, is a directed walk whose vertices are distinct. Paths of length zero are allowed. 
        A \emph{directed cycle}, or simply a \emph{cycle}, is a closed directed walk of length at least two in which no vertex is repeated except the first and last vertices.
        In particular, two opposite arcs form a directed cycle of length two. A digraph is \emph{acyclic} if it contains no directed cycle. The \emph{digirth} $g(D)$ is the length of a shortest directed cycle in $D$, and we set $g(D)=+\infty$ if $D$ is acyclic.

		Following~\cite{chen2026}, we denote the initial and terminal vertices of a path $P$ by $\init(P)$ and $\term(P)$, respectively. A \emph{$(u,v)$-path} is a path with initial vertex $u$ and terminal vertex $v$. The length of $P$ is $\|P\|=|A(P)|$. If $P$ and $Q$ are paths such that $V(P)\cap V(Q)=\{\term(P)\}=\{\init(Q)\}$, then $P\circ Q$ denotes the path obtained by concatenating $P$ and $Q$.

		For vertices $u,v$ on a path $P$, with $u$ preceding $v$ or $u=v$, let $P[u,v]$ denote the $(u,v)$-subpath of $P$. For vertices $u,v$ on a directed cycle $C$, let $C[u,v]$ denote the $(u,v)$-path obtained by following the direction of $C$, with $C[u,u]$ consisting only of $u$. The \emph{internal vertices} of a path are all its vertices other than its initial and terminal vertices. Paths are \emph{vertex-disjoint} if their vertex sets are pairwise disjoint, and \emph{internally vertex-disjoint} if any vertex shared by two paths is an endpoint of both. In particular, distinct internally vertex-disjoint $(u,v)$-paths have precisely $u$ and $v$ as common vertices. A subdivision of $C(k_1,k_2)$ is therefore the union of two internally vertex-disjoint paths with the same initial and terminal vertices, of lengths at least $k_1$ and $k_2$, respectively.

		For $u,v\in V(D)$, let $\dist_D(u,v)$ denote the length of a shortest $(u,v)$-path in $D$, and set $\dist_D(u,v)=+\infty$ if no such path exists. In particular, $\dist_D(v,v)=0$. We say that $v$ is \emph{reachable} from $u$ if $D$ contains a $(u,v)$-path. A digraph is \emph{strongly connected} if every vertex is reachable from every other vertex. A \emph{strongly connected component} is a maximal strongly connected induced subdigraph. It is \emph{terminal} if no arc leaves its vertex set.

		For distinct vertices $u,v\in V(D)$, a set $S\subseteq V(D)\setminus\{u,v\}$ is a \emph{$(u,v)$-separator} if $D-S$ contains no $(u,v)$-path. A vertex $c\in V(D)\setminus\{u,v\}$ is a \emph{$(u,v)$-cut-vertex} if $D$ contains a $(u,v)$-path but $D-c$ does not.

		By an \emph{embedding} $\varphi$ of $F$ into $D$, we mean an injective map $\varphi:V(F)\to V(D)$ such that $\varphi(u)\varphi(v)\in A(D)$ for every $uv\in A(F)$. A \emph{partial embedding} is an injective map $\varphi:X\to V(D)$ for some $X\subseteq V(F)$ that preserves every arc of $F[X]$. We write $\dom(\varphi)=X$ for its domain and $\varphi(Y)=\{\varphi(y):y\in Y\}$ for the image of $Y\subseteq X$. A partial embedding $\psi$ \emph{extends} $\varphi$ if $\dom(\varphi)\subseteq\dom(\psi)$ and $\psi(v)=\varphi(v)$ for every $v\in\dom(\varphi)$. Copies and embeddings in this paper are not required to be induced.

		\subsection{Auxiliary results}
		
		\begin{proposition}\label{replace-antichain}
			Let $F$ be an out-forest, let $A$ be an antichain of $F$, and let $v\in A$ be an internal vertex. If $C_F(v)$ denotes the set of children of $v$, then
			$$
			A'=(A\setminus\{v\})\sqcup C_F(v)
			$$
			is also an antichain of $F$.
		\end{proposition}
		
		\begin{proof}
			The vertices of $C_F(v)$ are pairwise incomparable. It remains to show that every vertex of $A\setminus\{v\}$ is incomparable with every vertex of $C_F(v)$.
			
			Let $u\in A\setminus\{v\}$ and $c\in C_F(v)$. If $u$ and $c$ were comparable, then either $u$ would be an ancestor of $c$, in which case $u$ and $v$ would be comparable, or $c$ would be an ancestor of $u$, in which case $v$ would be an ancestor of $u$. Both cases contradict the fact that $A$ is an antichain.
		\end{proof}
		
		\begin{proposition}\label{antichain-leaves}
			Let $F$ be an out-forest and let $A$ be an antichain of $F$. Then
			$$
			|A|\le L(F).
			$$
		\end{proposition}
		
		\begin{proof}
			For each $v\in A$, choose a leaf $\ell_v$ that is a descendant of $v$. We claim that these leaves are pairwise distinct. Suppose that $\ell_u=\ell_v$ for two distinct vertices $u,v\in A$. Then both $u$ and $v$ lie on the unique path from the root of their component to $\ell_u$. Hence one of $u$ and $v$ is an ancestor of the other, contradicting the fact that $A$ is an antichain. Thus $|A|\le L(F)$.
		\end{proof}
		
		We also record the following elementary facts.

		\begin{proposition}\label{walk-cycle}
		    Every positive-length closed directed walk contains a directed cycle of length at most the length of the walk.
		\end{proposition}
		
		The following is the directed vertex version of Menger's theorem; see~\cite[Theorem~5.4.1(b)]{bangjensen2009}.

		\begin{theorem}[Menger's theorem~\cite{menger1927}]\label{directed-menger}
		    Let $D$ be a digraph, and let $u,v$ be distinct vertices with $uv\notin A(D)$. Then the maximum number of pairwise internally vertex-disjoint $(u,v)$-paths equals the minimum cardinality of a $(u,v)$-separator.
		\end{theorem}

	\section{The proof of Theorem~\ref{2-block}}
		
		\noindent
		
		In this section, we prove Theorem~\ref{2-block}.
		
		
		\begin{proof}[Proof of Theorem~\ref{2-block}]
			Let $H$ be a terminal strongly connected component of $D$. Since no arc leaves $H$, we have $d_H^+(v)=d_D^+(v)$ for every $v\in V(H)$. Hence $d_H^+(v)\ge2$ for every $v\in V(H)\setminus\{v_0\}$, and if $v_0\in V(H)$, then $d_H^+(v_0)\ge1$.
			
			Set $h=v_0$ if $v_0\in V(H)$, and otherwise choose an arbitrary vertex $h\in V(H)$. Then every vertex of $H$ other than $h$ has out-degree at least two. Moreover, since $H$ is strongly connected, $N_H^+(V(H)\setminus\{h\})=\{h\}$.
			
			Among all nonempty sets $W\subseteq V(H)\setminus\{h\}$ with $|N_H^+(W)|=1$, choose one of minimum cardinality and denote it by $U$. Write $N_H^+(U)=\{x\}$.
			
			\begin{claim}\label{HU-strongly-connected}
				The digraph $H[U]$ is strongly connected.
			\end{claim}
			
			\begin{proof}
				Let $K$ be the vertex set of a terminal strongly connected component of $H[U]$. Since no arc of $H[U]$ leaves $K$, we have $N_H^+(K)\subseteq N_H^+(U)=\{x\}$. Moreover, $N_H^+(K)\ne\varnothing$, since $H$ is strongly connected and $K\subsetneq V(H)$. Hence $N_H^+(K)=\{x\}$.
				
				By the minimality of $U$, we must have $K=U$. Therefore, $H[U]$ is strongly connected.
			\end{proof}
			
			\begin{claim}\label{avoid-one-vertex}
				For any distinct vertices $v,c\in U$, the digraph $H-c$ contains a $(v,x)$-path.
			\end{claim}
			
			\begin{proof}
				Suppose otherwise, and let $R$ be the set of vertices reachable from $v$ in $H-c$. Since $x$ is not reachable from $v$ in $H-c$ and every arc leaving $U$ enters $x$, we have $\varnothing\ne R\subseteq U\setminus\{c\}$. Moreover, $N_H^+(R)\subseteq\{c\}$. Since $H$ is strongly connected and $R\subsetneq V(H)$, we have $N_H^+(R)\ne\varnothing$. Hence $N_H^+(R)=\{c\}$.
				This contradicts the minimality of $U$.
			\end{proof}
			
			Choose an $(x,r)$-path $Q$ in $H$ for some $r\in U$ such that $V(Q)\cap U=\{r\}$. Since $N_H^+(U)=\{x\}$, every vertex of $U$ has at most one out-neighbor outside $U$. As every vertex of $U$ has out-degree at least two in $H$, we have $\delta^+(H[U])\ge1$, and hence we may choose an arc $ru\in A(H[U])$.
			
			We claim that there exist two internally vertex-disjoint $(u,r)$-paths in $H$. Suppose first that $ur\in A(H)$. By Claim~\ref{avoid-one-vertex}, there exists a $(u,x)$-path $P$ in $H-r$. Since $N_H^+(U)=\{x\}$, all vertices of $P$ except $x$ lie in $U$, while all internal vertices of $Q$ lie outside $U$. Hence $P\circ Q$ is a $(u,r)$-path. Together with the arc $ur$, this gives two internally vertex-disjoint $(u,r)$-paths.
			
			Now suppose that $ur\notin A(H)$. We show that $H$ has no $(u,r)$-cut-vertex. Let $c\in V(H)\setminus\{u,r\}$. If $c\notin U$, then, since $H[U]$ is strongly connected, there is a $(u,r)$-path in $H[U]$, which avoids $c$. If $c\in U\setminus\{u,r\}$, then by Claim~\ref{avoid-one-vertex}, $H-c$ contains a $(u,x)$-path $P$. If $P$ contains $r$, then $P[u,r]$ is a $(u,r)$-path in $H-c$. Otherwise, $P\circ Q$ is a $(u,r)$-path in $H-c$.
			
			By Theorem~\ref{directed-menger}, $H$ contains two internally vertex-disjoint $(u,r)$-paths. Together with the arc $ru$, these paths form two directed cycles intersecting exactly in $ru$.
			
			Since $u,r\in U\subseteq V(H)\setminus\{h\}$, we have $u,r\ne h$. If $v_0\in V(H)$, then $h=v_0$, and hence $u,r\ne v_0$. If $v_0\notin V(H)$, then $u,r\in V(H)$ also implies $u,r\ne v_0$. Therefore, both endpoints of the common arc $ru$ are different from $v_0$.
		\end{proof}

		\section{Out-forests}
		
		\noindent
		
		In this section, we first prove Theorem~\ref{out-forest} and then give several applications of the theorem.
		
		\subsection{Proof of Theorem~\ref{out-forest}}
		
		\noindent
		
		
		\begin{proof}[Proof of Theorem~\ref{out-forest}]
			We construct the desired embedding step by step. 
            Let $\rho_1, \cdots,\rho_s$ be vertices chosen randomly from $V(D)$. 
            First, we embed the component roots. Let $A_0=\{r_1,\ldots,r_s\}$ and $Z_0=\varnothing$, and define the partial embedding $\varphi_0$ by
			$$
			\dom(\varphi_0)=\{r_1,\ldots,r_s\}
			$$
			and
			$$
			\varphi_0(r_j)=\rho_j
			\qquad\text{for every }j\in[s].
			$$
			
			At each step $i$, we maintain a partial embedding $\varphi_i$, a set $A_i\subseteq\dom(\varphi_i)$,  and a set $Z_i\subseteq V(D)$ such that the following properties hold.
			
			\begin{enumerate}
				\item\label{antichain} $A_i$ is an antichain of $F$ such that every path in $F$ from a component root to a leaf contains exactly one vertex of $A_i$.
				
				\item\label{domain} $\dom(\varphi_i)$ consists exactly of the vertices that are ancestors of some vertex in $A_i$, and $Z_i=\varphi_i\bigl(\dom(\varphi_i)\setminus A_i\bigr)$.
				In particular, $\varphi_i\bigl(\dom(\varphi_i)\bigr)=\varphi_i(A_i)\sqcup Z_i$.
				
				\item\label{invariant} For every $v\in A_i$ and every $z\in Z_i$,
				$$
				\dist_D(\varphi_i(v),z)>h_F(v).
				$$
			\end{enumerate}
			
			We first verify that $A_0, Z_0, \varphi_0$ satisfy properties~\ref{antichain}--\ref{invariant}. The component roots form an antichain and meet every path in $F$ from a component root to a leaf exactly once, so property~\ref{antichain} holds. Property~\ref{domain} follows from $\dom(\varphi_0)=A_0$ and $Z_0=\varnothing$. Property~\ref{invariant} holds vacuously since $Z_0=\varnothing$.
		
			Now suppose that, for some $k\ge0$, the sets $A_k$, $Z_k$, and the partial embedding $\varphi_k$ have been constructed and satisfy properties~\ref{antichain}--\ref{invariant}. If $\dom(\varphi_k)=V(F)$, then $\varphi_k$ is already the required embedding. We may therefore assume that $\dom(\varphi_k)\ne V(F)$. Our goal is to construct $A_{k+1}$, $Z_{k+1}$, and a partial embedding $\varphi_{k+1}$ extending $\varphi_k$, and to verify that they satisfy properties~\ref{antichain}--\ref{invariant}.
			
			\begin{claim}\label{A_k internal}
				$A_k$ contains an internal vertex.
			\end{claim}
			
			\begin{proof}
				Suppose not. Then every vertex of $A_k$ is a leaf. By property~\ref{antichain}, every leaf of $F$ belongs to $A_k$. It then follows from property~\ref{domain} that every vertex on a path in $F$ from a component root to a leaf belongs to $\dom(\varphi_k)$, and hence $\dom(\varphi_k)=V(F)$, a contradiction.
			\end{proof}
			
			We next choose an internal vertex in $A_k$ from which the embedding can be extended. Define an auxiliary digraph $H_k$ whose vertex set consists of the internal vertices in $A_k$. By Claim~\ref{A_k internal}, $V(H_k)$ is nonempty. For distinct $v,w\in V(H_k)$, put an arc $v\to w$ in $H_k$ whenever
			$$
			\dist_D(\varphi_k(v),\varphi_k(w))\le h_F(v).
			$$
			
			\begin{claim}\label{H_k acyclic}
				$H_k$ is acyclic.
			\end{claim}
			
			\begin{proof}
				Suppose not, and let $v_1\to v_2\to\cdots\to v_t\to v_1$ be a directed cycle in $H_k$. For each $j\in[t]$, there is a $(\varphi_k(v_j),\varphi_k(v_{j+1}))$-path in $D$ of length at most $h_F(v_j)$, where the indices are taken modulo $t$. Concatenating these paths gives a closed directed walk in $D$ of length at most $\sum\limits_{j=1}^{t}h_F(v_j)$. Since $v_1,\ldots,v_t$ form an antichain of $F$, we have
				$$
				\sum_{j=1}^{t}h_F(v_j)\le B(F).
				$$
				By Proposition~\ref{walk-cycle}, $D$ contains a directed cycle of length at most $B(F)$, contradicting $g(D)>B(F)$.
			\end{proof}
			
			Since $H_k$ is acyclic, it has a source. Choose a source $v\in V(H_k)$, so that $d^-_{H_k}(v)=0$. Let $C_F(v)$ denote the set of children of $v$ in $F$.
			
			\begin{claim}\label{C_F}
				No vertex of $C_F(v)$ belongs to $\dom(\varphi_k)$.
			\end{claim}
			
			\begin{proof}
				Suppose that some $c\in C_F(v)$ belongs to $\dom(\varphi_k)$. By property~\ref{domain}, $c$ is an ancestor of some $a\in A_k$. Since $v$ is the parent of $c$, it follows that $v$ is an ancestor of $a$. As $v,a\in A_k$, this contradicts the fact that $A_k$ is an antichain by property~\ref{antichain}.
			\end{proof}
			
			By Claim~\ref{C_F}, $C_F(v)\cap A_k=\varnothing$. By property~\ref{domain}, $\varphi_k(v)\notin Z_k$. We now define
			$$
			A_{k+1}=(A_k\setminus\{v\})\sqcup C_F(v),
			\qquad
			Z_{k+1}=Z_k\sqcup\{\varphi_k(v)\}.
			$$
			We first verify property~\ref{antichain} for $A_{k+1}$.
			
			By Proposition~\ref{replace-antichain}, $A_{k+1}$ is an antichain. Let $P$ be a path in $F$ from a component root to a leaf. If $v\notin P$, then $P$ meets $A_{k+1}$ at the same unique vertex at which it meets $A_k$. If $v\in P$, then $P$ contains exactly one child of $v$, which replaces $v$ as the unique vertex of $A_{k+1}$ on $P$. Hence every such path contains exactly one vertex of $A_{k+1}$, and property~\ref{antichain} holds.
			
			\begin{claim}\label{embed}
				There are at least $|C_F(v)|$ out-neighbors of $\varphi_k(v)$ outside $\varphi_k(\dom(\varphi_k))$.
			\end{claim}
			
			\begin{proof}
				Since $A_{k+1}$ is an antichain, Proposition~\ref{antichain-leaves} gives $|A_{k+1}|\le L(F)$. Together with $A_{k+1}=(A_k\setminus\{v\})\sqcup C_F(v)$, this yields 
				$$
				|C_F(v)|\le L(F)-|A_k|+1.
				$$
				
				We next observe that $N_D^+(\varphi_k(v))\cap Z_k=\varnothing$. Otherwise, some $z\in Z_k$ is an out-neighbor of $\varphi_k(v)$, so $\dist_D(\varphi_k(v),z)=1\le h_F(v)$, contradicting property~\ref{invariant}.
				
				By property~\ref{domain}, $\varphi_k(\dom(\varphi_k))= \varphi_k(A_k)\sqcup Z_k$. Among the out-neighbors of $\varphi_k(v)$, none belongs to $Z_k$, and at most $|A_k|-1$ belong to $\varphi_k(A_k\setminus\{v\})$. Moreover, $\varphi_k(v)\notin N_D^+(\varphi_k(v))$. Therefore
				\begin{align*}
					\bigl|N_D^+(\varphi_k(v))
					\setminus\varphi_k(\dom(\varphi_k))\bigr|
					&\ge d_D^+(\varphi_k(v))-(|A_k|-1)\\
					&\ge L(F)-|A_k|+1\\
					&\ge |C_F(v)|. \qedhere
				\end{align*}
			\end{proof}
			
			By Claim~\ref{embed}, there are at least $|C_F(v)|$ out-neighbors of $\varphi_k(v)$ outside $\varphi_k(\dom(\varphi_k))$. Hence we may choose a set
			$$
			X=\{x_c:c\in C_F(v)\}
			\subseteq
			N_D^+(\varphi_k(v))\setminus\varphi_k(\dom(\varphi_k))
			$$
			with $|X|=|C_F(v)|$.
			
			We then define
			$$
			\dom(\varphi_{k+1})
			=
			\dom(\varphi_k)\sqcup C_F(v)
			$$
			and
			$$
			\varphi_{k+1}(u)=
			\begin{cases}
				\varphi_k(u), & u\in\dom(\varphi_k),\\
				x_u, & u\in C_F(v).
			\end{cases}
			$$
			By Claim~\ref{C_F}, this definition is well defined. Since the vertices in $X$ are distinct, lie outside $\varphi_k(\dom(\varphi_k))$, and are out-neighbors of $\varphi_k(v)$, the map $\varphi_{k+1}$ is a partial embedding extending $\varphi_k$.
			
			For verifying property~\ref{domain}, all ancestors of each $a\in A_k\setminus\{v\}$ already belong to $\dom(\varphi_k)$. For each $c\in C_F(v)$, the ancestors of $c$ are precisely the ancestors of $v$ together with $c$ itself. Hence $\dom(\varphi_{k+1})$ consists exactly of the vertices that are ancestors of some vertex in $A_{k+1}$. Moreover, $\dom(\varphi_{k+1})\setminus A_{k+1}=\bigl(\dom(\varphi_k)\setminus A_k\bigr)\sqcup\{v\}$. Therefore, using property~\ref{domain} for $\varphi_k$,
			$$
			\begin{aligned}
				\varphi_{k+1}\bigl(\dom(\varphi_{k+1})\setminus A_{k+1}\bigr)
				&=
				\varphi_k\bigl(\dom(\varphi_k)\setminus A_k\bigr)
				\sqcup\{\varphi_k(v)\}\\
				&=
				Z_k\sqcup\{\varphi_k(v)\}\\
				&=
				Z_{k+1}.
			\end{aligned}
			$$
			Thus property~\ref{domain} holds for $A_{k+1}$, $Z_{k+1}$, and $\varphi_{k+1}$.
			
			We now verify property~\ref{invariant}. Let $w\in A_{k+1}$ and $z\in Z_{k+1}$. We distinguish four cases.
			
			\medskip
			\noindent\textbf{Case 1.} $w\in A_k\setminus\{v\}$ and $z\in Z_k$.
			
			Since $\varphi_{k+1}$ extends $\varphi_k$, property~\ref{invariant} at step $k$ gives
			$\dist_D(\varphi_{k+1}(w),z)=\dist_D(\varphi_k(w),z)>h_F(w)$.
			
			\medskip
			\noindent\textbf{Case 2.} $w\in A_k\setminus\{v\}$ and $z=\varphi_k(v)$.
			
			If $w$ is internal, then $w\in V(H_k)$. Since $v$ is a source of $H_k$, there is no arc $w\to v$ in $H_k$. By the definition of $H_k$, $\dist_D(\varphi_k(w),\varphi_k(v))>h_F(w)$. If $w$ is a leaf, then $h_F(w)=0$, and since $\varphi_k$ is injective and $w\ne v$, we have $\dist_D(\varphi_k(w),\varphi_k(v))>0=h_F(w)$. Thus in either case, $\dist_D(\varphi_{k+1}(w),z)>h_F(w)$.
			
			\medskip
			\noindent\textbf{Case 3.} $w=c\in C_F(v)$ and $z\in Z_k$.
			
			Since $c$ is a child of $v$, $1+h_F(c)\le h_F(v)$. Suppose that $\dist_D(\varphi_{k+1}(c),z)\le h_F(c)$. Since $\varphi_k(v)\varphi_{k+1}(c)\in A(D)$, we obtain
			$\dist_D(\varphi_k(v),z)\le \dist_D(\varphi_k(v),\varphi_{k+1}(c))+\dist_D(\varphi_{k+1}(c),z)\le 1+h_F(c)\le h_F(v)$,
			contradicting property~\ref{invariant} at step $k$. Hence $\dist_D(\varphi_{k+1}(c),z)>h_F(c)$.
			
			\medskip
			

            \noindent\textbf{Case 4.} $w=c\in C_F(v)$ and $z=\varphi_k(v)$.
            
Suppose that
$\dist_D(\varphi_{k+1}(c),\varphi_k(v))\le h_F(c)$.
Then a $(\varphi_{k+1}(c),\varphi_k(v))$-path of length at most
$h_F(c)$, together with the arc $\varphi_k(v)\varphi_{k+1}(c)$,
forms a directed cycle of length at most
\[
1+h_F(c)\le h_F(v)\le B(F),
\]
contradicting $g(D)>B(F)$.
Therefore,
$\dist_D(\varphi_{k+1}(c),\varphi_k(v))>h_F(c)$.
			
			Hence properties~\ref{antichain}--\ref{invariant} hold for $A_{k+1}$, $Z_{k+1}$, and $\varphi_{k+1}$. Therefore, the construction can be repeated as long as $\dom(\varphi_k)\ne V(F)$.
			
			At each step, since $v$ is internal, $C_F(v)\ne\varnothing$, and hence
			$$
			\dom(\varphi_k)\subsetneq\dom(\varphi_{k+1}).
			$$
			Since $F$ is finite, the process terminates after finitely many steps. Thus, for some $m$,
			$$
			\dom(\varphi_m)=V(F).
			$$
			Consequently, $\varphi_m$ is an embedding of $F$ into $D$. Moreover, each $\varphi_{k+1}$ extends $\varphi_k$, so
			$$
			\varphi_m(r_j)=\varphi_0(r_j)=\rho_j
			\qquad\text{for every }j\in[s].
			$$
			This completes the proof.
		\end{proof}
		
		\subsection{Applications}\label{applications}
		
		\noindent
		
		We now derive several applications of Theorem~\ref{out-forest}.
		
		Let $A$ be an antichain of $F$. For each $v\in A$, choose a descendant leaf $\ell_v$ with $\dist_F(v,\ell_v)=h_F(v)$, and let $P_v$ be the unique $(v,\ell_v)$-path in $F$. The $h_F(v)$ vertices of $V(P_v)\setminus\{\ell_v\}$ are all internal vertices of $F$. Since the descendant subtrees rooted at distinct vertices of $A$ are vertex-disjoint, it follows that
		$$
		\sum_{v\in A}h_F(v)=\sum_{v\in A}|V(P_v)\setminus\{\ell_v\}|\le |V(F)|-L(F).
		$$
		Taking the maximum over all antichains $A$, we obtain
		$$
		B(F)\le |V(F)|-L(F).
		$$
	Combining the bound above with Theorem~\ref{out-forest}, we obtain the following general result for out-forests with given order and number of leaves.
		
		\begin{corollary}\label{coro1}
			Let $F$ be an out-forest on $n$ vertices with $L$ leaves and component roots $r_1,\ldots,r_s$. If a digraph $D$ satisfies
			$$
			\delta^+(D)\ge L
			\qquad\text{and}\qquad
			g(D)\ge n-L+1,
			$$
			then for any distinct vertices $\rho_1,\ldots,\rho_s\in V(D)$, there exists an embedding $\varphi:F\hookrightarrow D$ such that $\varphi(r_i)=\rho_i$ for every $i\in[s]$.
		\end{corollary}
		
		We next turn to several special families of out-forests for which the parameter $B(F)$ can be determined explicitly. First, suppose that $F$ is the disjoint union of directed paths of lengths $m_1,\ldots,m_k\ge0$. Then $L(F)=k$ and $B(F)=\sum\limits_{i=1}^k m_i$. Applying Theorem~\ref{out-forest}, we obtain the following.
		
		\begin{corollary}\label{coro2}
			Let $k\ge1$ and let $m_1,\ldots,m_k$ be nonnegative integers. Let $D$ be a digraph satisfying
			$$
			\delta^+(D)\ge k
			\qquad\text{and}\qquad
			g(D)>\sum_{i=1}^k m_i.
			$$
			Then, for any distinct vertices $x_1,\ldots,x_k\in V(D)$, there exist pairwise vertex-disjoint directed paths $P_1,\ldots,P_k$ such that, for every $i\in[k]$, the path $P_i$ starts at $x_i$ and has length $m_i$.
		\end{corollary}
		
		We next consider subdivisions of out-stars. Let $F$ consist of $k\ge2$ internally vertex-disjoint directed paths $P_1,\ldots,P_k$ sharing exactly their initial vertex $r$, where $P_i$ has length $m_i\ge1$ for every $i\in[k]$. Clearly, $L(F)=k$. We claim that
		$$
		B(F)=\max\left\{\max_{i\in[k]}m_i,\,\sum_{i=1}^k(m_i-1)\right\}.
		$$
		Indeed, let $A$ be an antichain of $F$. If $r\in A$, then $A=\{r\}$, since every other vertex of $F$ is a descendant of $r$. Hence
		$$
		\sum_{v\in A}h_F(v)=h_F(r)=\max_{i\in[k]}m_i.
		$$
		If $r\notin A$, then $A$ contains at most one vertex from each directed path $P_i$, since any two vertices on the same directed path are comparable. For each $i\in[k]$, the maximum height of a vertex of $P_i$ other than $r$ is $m_i-1$, attained by the successor of $r$ on $P_i$. Hence
		$$
		\sum_{v\in A}h_F(v)\le\sum_{i=1}^k(m_i-1).
		$$
		Equality holds when $A$ consists of the successors of $r$ on $P_1,\ldots,P_k$. This proves the claimed formula for $B(F)$.
		
		Therefore, Theorem~\ref{out-forest} yields the following result.
		
		\begin{corollary}\label{coro3}
			Let $F$ consist of $k\ge2$ internally vertex-disjoint directed paths $P_1,\ldots,P_k$ sharing exactly their initial vertex, where $P_i$ has length $m_i\ge1$ for every $i\in[k]$. If a digraph $D$ satisfies
			$$
			\delta^+(D)\ge k
			\qquad\text{and}\qquad
			g(D)>\max\left\{\max_{i\in[k]}m_i,\,\sum_{i=1}^k(m_i-1)\right\},
			$$
			then for every vertex $x\in V(D)$, $D$ contains a copy of $F$ with center $x$.
		\end{corollary}
		
		We finally consider the digraph $S_k^{(\ell)}$ studied by Picasarri-Arrieta and Rambaud~\cite{picasarri2024}. Recall that $S_k^{(\ell)}$ consists of $k$ internally vertex-disjoint directed paths of length $\ell$ sharing exactly their initial vertex. Taking $m_1=\cdots=m_k=\ell$ in Corollary~\ref{coro3}, we immediately obtain the following.
		
		\begin{corollary}\label{coro4}
			Let $k,\ell\ge2$, and let $D$ be a digraph satisfying
			$$
			\delta^+(D)\ge k
			\qquad\text{and}\qquad
			g(D)\ge k(\ell-1)+1.
			$$
			Then, for every vertex $x\in V(D)$, $D$ contains a copy of $S_k^{(\ell)}$ with center $x$.
		\end{corollary}
		
		In~\cite{chen2026}, Chen et al. constructed a family of digraphs to establish a lower bound for subdivisions of $C(k,k)$. By modifying their construction, we obtain a family of digraphs showing that the digirth bounds in Corollaries~\ref{coro2} and~\ref{coro4} are sharp.
		
		\medskip
		\noindent\textbf{Construction $D(m_1,\ldots,m_k)$.}
		
		Let $k\ge1$ and let $m_1,\ldots,m_k$ be nonnegative integers with $M=m_1+\cdots+m_k\ge2$. After relabeling, assume that $m_1,\ldots,m_r>0$ and $m_{r+1}=\cdots=m_k=0$.
		
		For each $t\in\{0,\ldots,M-1\}$, take a directed cycle $C_t$ of length $M$ and an independent set $V_t$ of size $k-1$, where all these cycles and sets are pairwise vertex-disjoint. Put $B_t=V(C_t)\sqcup V_t$, and take the index $t$ modulo $M$ throughout the construction.
		
		For each $t$, we first construct a set $X_t$ of $k$ vertices that will serve as the common out-neighborhood of the vertices in $V_t$. If $r\ge2$, choose vertices $x_{t,1},\ldots,x_{t,r}$ on $C_t$ in cyclic order such that $\|C_t[x_{t,i},x_{t,i+1}]\|=m_i$ for $i\in[r-1]$ and $\|C_t[x_{t,r},x_{t,1}]\|=m_r$. When $r=1$, choose an arbitrary vertex $x_{t,1}\in V(C_t)$. For each $i\in\{r+1,\ldots,k\}$, choose a distinct vertex $x_{t,i}\in V_{t+1}$; this is possible since $k-r\le |V_{t+1}|=k-1$. Set $X_t=\{x_{t,1},\ldots,x_{t,k}\}$.
		
		Add all arcs from $V(C_t)$ to $V_{t+1}$ and all arcs from $V_t$ to $X_t$ for every $t$. There are no other arcs. Denote the resulting digraph by $D(m_1,\ldots,m_k)$ ($D(2,1,0)$ is shown in Figure~\ref{fig:D-construction}).

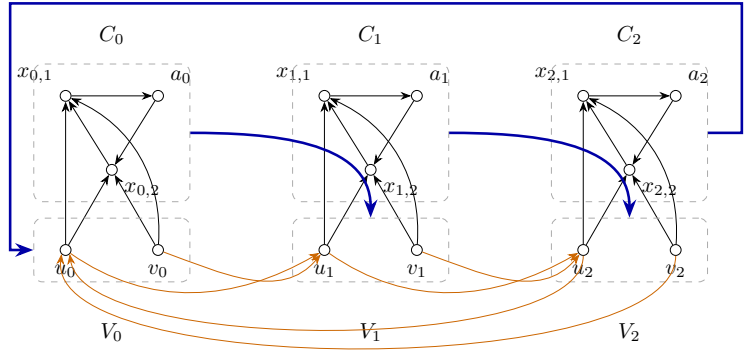
\begin{figure}[htbp]
\centering
\resizebox{0.60\textwidth}{!}{%
\begin{tikzpicture}[
    x=1cm,y=1cm,
    >=Stealth,
    vertex/.style={
        circle,draw,fill=white,
        minimum size=5pt,inner sep=0pt
    },
    arc/.style={->,thin},
    complete/.style={
        ->,blue!65!black,line width=1.2pt
    },
    cross/.style={
        ->,orange!80!black,thin
    },
    group/.style={
        draw=gray!65,dashed,rounded corners,
        inner sep=12pt
    },
    every node/.style={font=\small}
]

\foreach \t in {0,1,2}{
    \pgfmathsetmacro{\xx}{4.2*\t}

    \node[vertex,label=above left:{$x_{\t,1}$}]
        (p\t) at (\xx,2.5) {};
    \node[vertex,label=above right:{$a_{\t}$}]
        (a\t) at (\xx+1.5,2.5) {};
    \node[vertex,label=below right:{$x_{\t,2}$}]
        (q\t) at (\xx+0.75,1.3) {};

    \draw[arc] (p\t) -- (a\t);
    \draw[arc] (a\t) -- (q\t);
    \draw[arc] (q\t) -- (p\t);

    \node[vertex,label=below:{$u_{\t}$}]
        (u\t) at (\xx,0) {};
    \node[vertex,label=below:{$v_{\t}$}]
        (v\t) at (\xx+1.5,0) {};

    \draw[arc] (u\t) -- (p\t);
    \draw[arc] (u\t) -- (q\t);
    \draw[arc] (v\t) -- (q\t);
    \draw[arc] (v\t) to[bend right=35] (p\t);

    \node[group,fit=(p\t)(a\t)(q\t)] (C\t) {};
    \node[group,fit=(u\t)(v\t)] (V\t) {};

    \node[above=5pt] at (C\t.north) {$C_{\t}$};
    \node[below=15pt] at (V\t.south) {$V_{\t}$};
}

\foreach \t/\s in {0/1,1/2}{
    \draw[complete]
        (C\t.east)
        to[out=0,in=90]
        (V\s.north);

    \draw[cross]
        (u\t) to[out=-35,in=-150] (u\s);
    \draw[cross]
        (v\t) to[out=-20,in=-130] (u\s);
}

\draw[complete]
    (C2.east)
    -- ++(0.55,0)
    |- (5,4.0)
    -| (-0.9,0)
    -- (V0.west);

\draw[cross]
    (u2.south)
    .. controls ($(u2.south)+(0,-1.6)$)
             and ($(u0.south)+(0.3,-1.6)$)
    .. (u0.south east);

\draw[cross]
    (v2.south)
    .. controls ($(v2.south)+(0,-2.0)$)
             and ($(u0.south)+(-0.3,-2.0)$)
    .. (u0.south west);

\end{tikzpicture}%
}
\caption{The construction $D(2,1,0)$, with
$X_t=\{x_{t,1},x_{t,2},u_{t+1}\}$.
Each thick blue arrow represents all arcs from its initial
vertex set to its terminal vertex set.
Orange arrows represent the arcs from $V_t$ to $u_{t+1}$.
Indices are taken modulo $3$.}
\label{fig:D-construction}
\end{figure}

		\begin{claim}\label{claim5}
			The digraph $D(m_1,\ldots,m_k)$ satisfies
			$$
			\delta^+\bigl(D(m_1,\ldots,m_k)\bigr)=k,
			\qquad
			g\bigl(D(m_1,\ldots,m_k)\bigr)=M.
			$$
		\end{claim}
		
		\begin{proof}
			Every vertex of $C_t$ has exactly its successor on $C_t$ and the $k-1$ vertices of $V_{t+1}$ as its out-neighbors, while every vertex of $V_t$ has exactly the $k$ vertices of $X_t$ as its out-neighbors. Hence every vertex of $D(m_1,\ldots,m_k)$ has out-degree $k$, and therefore $\delta^+\bigl(D(m_1,\ldots,m_k)\bigr)=k$.
			
			Since each $C_t$ is a directed cycle of length $M$, we have $g\bigl(D(m_1,\ldots,m_k)\bigr)\le M$. It remains to show that every directed cycle has length at least $M$. A directed cycle contained in a single block $B_t$ must be $C_t$, since no vertex of $V_t$ has an in-neighbor in $B_t$. Moreover, every arc joining distinct blocks goes from $B_t$ to $B_{t+1}$. Hence any directed cycle using an arc between distinct blocks must visit all $M$ blocks before returning to its initial block, and therefore has length at least $M$. Thus $g\bigl(D(m_1,\ldots,m_k)\bigr)=M$.
		\end{proof}
		
		\begin{claim}\label{claim6}
			For every $t\in\{0,\ldots,M-1\}$, there do not exist pairwise vertex-disjoint directed paths $P_1,\ldots,P_k$ in $D(m_1,\ldots,m_k)$ such that $P_i$ starts at $x_{t,i}$ and has length $m_i$ for every $i\in[k]$. Consequently, the digirth bound in Corollary~\ref{coro2} is sharp whenever $m_1+\cdots+m_k\ge2$.
		\end{claim}
		
		\begin{proof}
			Fix $t$ and suppose that such paths $P_1,\ldots,P_k$ exist. For every $i>r$, we have $m_i=0$, so $P_i$ consists only of the vertex $x_{t,i}\in V_{t+1}$.
			
			We claim that each $P_i$ with $i\in[r]$ must leave $C_t$. Otherwise, suppose that $P_i\subseteq C_t$ for some $i\in[r]$. If $r\ge2$, then $P_i$ ends at $x_{t,i+1}$ when $i<r$, and at $x_{t,1}$ when $i=r$. In either case, $P_i$ meets the starting vertex of another path, a contradiction. If $r=1$, then $m_1=M=\|C_t\|$, so $P_1$ cannot be contained in $C_t$ and must leave $C_t$. Hence every $P_i$ with $i\in[r]$ leaves $C_t$.
			
			By the construction of $D(m_1,\ldots,m_k)$, the first vertex of $P_i$ outside $C_t$ lies in $V_{t+1}$. The $k-r$ paths of length zero already occupy the distinct vertices $x_{t,r+1},\ldots,x_{t,k}$ of $V_{t+1}$, leaving only
			$$
			|V_{t+1}|-(k-r)=(k-1)-(k-r)=r-1
			$$
			available vertices in $V_{t+1}$. However, the $r$ paths $P_1,\ldots,P_r$ must enter $V_{t+1}$ at distinct vertices, since they are pairwise vertex-disjoint. This is impossible.
		\end{proof}
		
		For $k,\ell\ge2$, let $D_{k,\ell}$ be the digraph $D(m_1,\ldots,m_k)$ with $m_i=\ell-1$ for every $i\in[k]$, so that $M=k(\ell-1)$. By Claim~\ref{claim5},
		$$
		\delta^+(D_{k,\ell})=k,
		\qquad
		g(D_{k,\ell})=k(\ell-1).
		$$
		
		\begin{claim}\label{claim7}
			  In $D_{k,\ell}$, no vertex in $\bigcup\limits_{t=0}^{M-1} V_t$ is the center of a copy of $S_k^{(\ell)}$. Consequently, the digirth bound in Corollary~\ref{coro4} is sharp.
		\end{claim}
		
		\begin{proof}
			Fix $v\in V_t$ and suppose that $v$ is the center of a copy of $S_k^{(\ell)}$. Since $N^+(v)=X_t=\{x_{t,1},\ldots,x_{t,k}\}$, the $k$ directed paths starting at $v$ must use all vertices of $X_t$ as their second vertices. Deleting $v$ yields $k$ pairwise vertex-disjoint directed paths of length $\ell-1$ starting at $x_{t,1},\ldots,x_{t,k}$, contradicting Claim~\ref{claim6}.
		\end{proof}

		\section{Remarks and discussions}
Fixed a digraph $H$, define the {\it girth-restricted Mader number} $\text{mader}^{(g)}_{\delta^+}(H)$ to be the least integer $d$ such that every digraph $D$ with $\delta^+(D)\ge d$ and  $g(D)\ge g$ contains a subdivision of $H$. 
In this paper, we prove that $\text{mader}^{(k+1)}_{\delta^+}(C(k,k))=2$. This implies that $\text{mader}^{(k)}_{\delta^+}(C(k,k))\ge 3$. A natural further question is to determine $\text{mader}^{(k)}_{\delta^+}(C(k,k))$.

Let $\Theta(\ell_1,\ldots,\ell_t)$ consist of $t$ internally disjoint directed $(s,z)$-paths with the indicated lengths (see Figure~\ref{fig:theta-construction}).

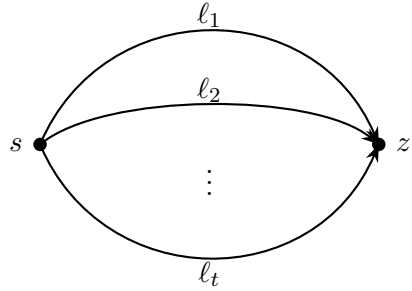
\begin{figure}[htbp]
\begin{center}
\resizebox{.36\linewidth}{!}{
\begin{tikzpicture}[x=1cm,y=1cm,>=Stealth,every node/.style={font=\normalsize},dot/.style={circle,fill=black,inner sep=1.8pt}]
\coordinate (s) at (-2.25,0);\coordinate (t) at (2.25,0);
\draw[thick,->] (s) .. controls (-1.4,2) and (1.4,2) .. (t);
\draw[thick,->] (s) .. controls (-1.4,.7) and (1.4,.7) .. (t);
\draw[thick,->] (s) .. controls (-1.4,-2) and (1.4,-2) .. (t);
\node at (0,1.72) {$\ell_1$};\node at (0,.72) {$\ell_2$};\node at (0,-1.72) {$\ell_t$};
\node at (0,-.4) {$\vdots$};
\node[dot,label=left:$s$] at (s) {};\node[dot,label=right:$z$] at (t) {};
\end{tikzpicture}
}
\end{center}
 \caption{$\Theta(\ell_1,\ldots,\ell_t)$}
\label{fig:theta-construction}
\end{figure}

\begin{problem}
Determine $\mader_{\delta^+}(\Theta(\ell_1,\ldots,\ell_t))$ for $t\ge3$.
\end{problem}

Let $n=|V(\Theta(\ell_1,\ldots,\ell_t))|=2+\sum_{i=1}^t(\ell_i-1)$. 
Clearly, the complete digraph on $n-1$ vertices has minimum out-degree $n-2$ and cannot contain a subdivision of $\Theta(\ell_1\ldots,\ell_t)$. Hence
\[
\mader_{\delta^+}\bigl(\Theta(\ell_1,\ldots,\ell_t)\bigr)
\ge |V(\Theta)|-1=1+\sum_{i=1}^t(\ell_i-1).
\]
For $t=2$ and $\ell_1,\ell_2\ge2$, equality holds.
\begin{problem}
When is this lower bound attained for $t\ge3$? 
\end{problem}


\subsection*{Acknowledgements} This work was supported by the National Key R\&D Program of China (2023YFA1010202, 2023YFA1010203), the
National Natural Science Foundation of China (12471336, 12501473), the Quantum Science and Technology-National Science and Technology Major Project(2021ZD0302902).

\subsection*{Declaration on the use of AI}
During the preparation of this article, the authors used ChatGPT-4o (OpenAI, used October 2026) to assist in discussing proof strategies, improving the exposition, and producing two schematic figures (TikZ code). The authors have reviewed and verified all proofs, arguments, and figures, and take full responsibility for the content of this publication.

		\vskip 3mm
	\end{spacing}
	
\end{document}